\documentclass{amsart}

\usepackage[latin1]{inputenc}
\usepackage[english]{babel}
\usepackage{indentfirst}
\usepackage{paralist}
\usepackage{amssymb}
\usepackage{amsthm}
\usepackage{amsmath}
\usepackage{mathrsfs}
\usepackage{amscd}
\usepackage{cite}
\usepackage[shortlabels]{enumitem} 
\usepackage{accents,fullpage}

\usepackage{xcolor}

\newcommand{\eqlab}[1]{\begin{equation}  \begin{aligned}#1 \end{aligned}\end{equation}} 
\newcommand{\bgs}[1]{\begin{equation*} \begin{aligned}#1\end{aligned}\end{equation*}} 

\newcommand{\R}{\ensuremath{\mathbb{R}}}

\newcommand{\N}{\ensuremath{\mathbb{N}}}

\DeclareMathOperator{\diam}{diam}

\newtheorem{theorem}{Theorem}[section]
\newtheorem{lemma}[theorem]{Lemma}
\newtheorem{corollary}[theorem]{Corollary}
\newtheorem{proposition}[theorem]{Proposition}

\theoremstyle{definition}

\newtheorem{example}[theorem]{Example}

\theoremstyle{remark}
\newtheorem{remark}[theorem]{Remark}

\numberwithin{equation}{section}

\begin{document}

\title[Obstructions to extension]{Obstructions to extension of Wasserstein distances for variable masses}


\author{Luca Lombardini}
\address{Universit\`a degli Studi di Padova, Dipartimento di Matematica ``Tullio Levi-Civita'', Via Trieste 63, 35121 Padova, Italy}

\author{Francesco Rossi}
\address{Universit\`a degli Studi di Padova, Dipartimento di Matematica ``Tullio Levi-Civita'', Via Trieste 63, 35121 Padova, Italy}




\begin{abstract}
	We study the possibility of defining a distance on the whole space of measures, with the property that the distance between two measures having the same mass is the Wasserstein distance, up to a scaling factor. We prove that, under very weak and natural conditions, if the base space is unbounded, then the scaling factor must be constant, independently of the mass. Moreover, no such distance can exist, if we include the zero measure. Instead, we provide examples with non-constant scaling factors for the case of bounded base spaces.
\end{abstract}

\maketitle


\section{Introduction and main results}\label{Main_res_Sect}

Wasserstein distances are crucial modern tools in mathematics, raising an enormous interest from the community of mathematical analysis (see e.g. \cite{villani2,AGS,caffarellireg}). Their applications are also extremely varied, ranging from crowd dynamics \cite{CPT,maury2}, to economics \cite{matthes,galichon}, to computer science \cite{peyre}. 

Yet, Wasserstein distances have several limitations, the most apparent being that they are defined between measures with the same mass only. This issue led to the definition of different possible generalizations of the Wasserstein distance between measures with different masses. The first attempt in this direction, described in \cite{figalli-gigli}, was related to the heat equation on a domain $\Omega$ with Dirichlet boundary condition, that clearly does not preserve mass. A second attempt, see \cite{figalli-arma}, was given by the so-called optimal partial transport problem: given two measures $\mu,\nu$ with different masses, one fixes a smaller mass and computes the Wasserstein distance between optimal submeasures.

A third group of contributions, known under the name of {\it generalized Wasserstein distance} \cite{gw,gw2}, was based on the optimization of a mixed Wasserstein-$L^1$ cost, in which a part of the measure is transported (with Wasserstein cost) and the remaining part is removed (with a $L^1$ cost). This contribution was then followed by several ``generalizations of the generalization'', see \cite{chizat,monsaingenon,liero}.

One of the main drawbacks of the distances given above is that they do not coincide, in general, with the standard Wasserstein distance when the two measures have the same mass. The only exception is clearly the partial optimal transportation problem when one chooses to transport the whole mass. This observation raises a natural question: {\it is it possible to define a distance between measures of different masses, that coincides with the Wasserstein distance in the case measures have the same mass?}

Our main result provides a negative answer, under quite general hypotheses. We show that it is not possible to extend the $p$-Wasserstein distance $W_p$ to a distance function $d$ defined on the whole space of measures in $\R^n$ and satisfying the property
\begin{equation}\label{fiber_condition_Wp}
d(\mu_1,\mu_2)=f(m)W_p\Big(\frac{\mu_1}{m},\frac{\mu_2}{m}\Big),
\end{equation}
whenever the mass $|\mu|:=\int_{\R^n}d\mu$ satisfies $|\mu_1|=|\mu_2|=m>0$.\\

More precisely, let $\mathcal M_c(\R^n)$ be the space of non-negative measures in $\R^n$ having finite mass and compact support and $\mathcal M^*_c(\R^n)$ the space of measures with positive finite mass and compact support. Then, $\mathcal M_c(\R^n)=\{0\}\cup \mathcal M^*_c(\R^n)$, where $0$ denotes the zero measure. Let also $\mathcal P_c(\R^n)$ be the space of probability measures on $\R^n$, i.e. of mass 1, having compact support. We use the notation $\R_{>0}$ for the space of positive real numbers $(0,+\infty)$, endowed with the Euclidean distance.
We stress that, as sets, we can identify
\eqlab{\label{set_isom}
\mathcal M^*_c(\R^n)\simeq \R_{>0}\times\mathcal P_c(\R^n),
}
via the bijection $\mu\mapsto(|\mu|,\mu/|\mu|)$.

We investigate the possibility of defining a distance function $d$ on $\mathcal M_c(\R^n)$ satisfying \eqref{fiber_condition_Wp}, i.e. with the property that on each fiber $\{m\}\times\mathcal P_c(\R^n)$ the distance $d$ coincides with the $p$-Wasserstein distance, up to a multiplicative scaling factor depending on $m>0$ only.

Our main result is a strong obstruction to a general definition of such a distance, as it provides a very weak condition forcing the function $f$ in \eqref{fiber_condition_Wp} to be constant. Indeed, we have the following first result.

\begin{theorem}[Unbounded space]\label{Teo_pWas}
	Let $p\in[1,+\infty)$ and let $d_\mathcal M:\mathcal M^*_c(\R^n)\times \mathcal M^*_c(\R^n)\to[0,+\infty)$ be a distance function satisfying the following two properties:
	\begin{enumerate}[(i)]
		\item there exists a function $f:\R_{>0}\to \R_{>0}$ such that
		\begin{equation*}
			d_\mathcal M(m\mu_1,m\mu_2)=f(m)W_p(\mu_1,\mu_2),
		\end{equation*}
		for every $m\in \R_{>0}$ and every $\mu_1,\mu_2\in \mathcal P_c(\R^n)$.
		\item There exists an unbounded subset $\Sigma\subseteq\R^n$ such that for every $m_0\in \R_{>0}$ there exist a radius $r>0$ and a positive constant $C>0$, both depending only on $m_0$, with
		\begin{equation}\label{in_unif_mass_cost_eq}
			\sup_{x\in \Sigma}d_\mathcal M(m_0\delta_x,m\delta_x)\leq C
		\end{equation}
		for every $m\in (m_0-r,m_0+r)$, where $\delta_x$ is the Dirac delta centered at $x$.
	\end{enumerate}
	Then, $f\equiv\lambda$, for some positive constant $\lambda>0$.
\end{theorem}	
	
	As a consequence, it is not possible to extend such a distance to the whole space $\mathcal M_c(\R^n)$, i.e. when the zero measure is added.
	
\begin{corollary}[Unbounded space with zero]\label{Cor_pWas}
	Let $d_\mathcal M$ be a distance function on $\mathcal M^*_c(\R^n)$ as in Theorem \ref{Teo_pWas}. Then, there exists no distance function $d_{\mathcal M_0}$ defined on $\mathcal M_c(\R^n)$, which agrees with $d_\mathcal M$ on $\mathcal M^*_c(\R^n)$, and such that either
	\begin{equation}\label{bded_0_eq}
		\sup_{x\in\Sigma}d_{\mathcal M_0}(0,m_0\delta_x)<+\infty,
	\end{equation}
	for some $m_0\in\R_{>0}$, or
	\begin{equation}\label{destr_mass_eq}
	\lim_{m\searrow0}d_{\mathcal M_0}(0,m\mu_i)=0,
	\end{equation}
	at least for two different $\mu_1\not=\mu_2\in\mathcal P_c(\R^n)$.
\end{corollary}

A few observations are in order. First of all, simple examples of distance functions on $\mathcal M^*_c(\R^n)$ satisfying \eqref{fiber_condition_Wp}, with $f$ constant, are given by the product metrics on the decomposed space \eqref{set_isom}, see Example \ref{exe_prod} below.
Roughly speaking, these distances measure separately the costs of changing mass and of transporting two equal masses, then sum up the two costs.

We also stress that condition \eqref{in_unif_mass_cost_eq} is very weak. It is an assumption regarding only the point masses located within $\Sigma$, and it represents the fact that the cost of ``destroying'' or ``creating'' mass, going from $m\delta_x$ to $m_0\delta_x$, is uniformly bounded with respect to the position $x\in\Sigma$, locally around $m_0$.
Furthermore, we point out that \eqref{in_unif_mass_cost_eq} is implied by other very natural conditions, like the invariance of $d_\mathcal M$ with respect to isometries together with the compatibility of $d_\mathcal M$ with weak convergence when varying the mass, see Proposition \ref{final_prop} for the precise statement.

In Theorem \ref{Teo_pWas} and Corollary \ref{Cor_pWas}, the specific properties of the $p$-Wasserstein distance actually play no role, except for the fact that $W_p$ is unbounded on $\mathcal M^*_c(\R^n)$---and in particular on the subset consisting of Dirac measures with base points belonging to $\Sigma$, an unbounded subset of $\R^n$.
As a consequence, Theorem \ref{Teo_pWas} and Corollary \ref{Cor_pWas} can be extended to much more general settings: in place of $\R^n$ we can consider a Polish space $(X, d_X)$ and in place of the Wasserstein distance $W_p$ we can consider any distance function $d_\mathcal P$ defined on the space of compactly supported probability measures $\mathcal P_c(X)$.

If $d_\mathcal P$ is unbounded, we can then translate Theorem \ref{Teo_pWas} to this setting in a straightforward way: in point (ii) simply consider a subset $S\subseteq\mathcal P_c(X)$ which is unbounded with respect to $d_\mathcal P$, in place of the set $\{\delta_x\,:\,x\in\Sigma\}$.

On the contrary, if $d_\mathcal P$ is bounded---e.g., if $d_\mathcal P$ is the bounded Lipschitz distance in $\mathcal P_c(\R^n)$ (see, e.g., \cite{Villani09,gw2}) or if $d_\mathcal P$ is the $p$-Wasserstein distance in $\mathcal P_c(X)$, when $(X,d_X)$ is bounded---the picture is completely different.

\begin{proposition}[Bounded space]\label{general_Teo}
	Let $(X,d_X)$ be a Polish space and let $d_\mathcal P:\mathcal P_c(X)\times \mathcal P_c(X)\to[0,+\infty)$ be a bounded distance function. Let $f:\R_{>0}\to\R_{>0}$ be a Lipschitz, increasing, function such that
	\begin{equation*}
		\sup_{\substack{m_1,m_2\in\R_{>0}\\ m_1\not=m_2}}\frac{|f(m_1)-f(m_2)|}{|m_1-m_2|}\leq\frac{1}{\diam \mathcal P_c(X)}\quad\mbox{and}\quad \lim_{m\searrow0}f(m)=0.
	\end{equation*}
	Then, the function $d_{\mathcal M_0}:\mathcal M_c(X)\times \mathcal M_c(X)\to[0,+\infty)$ defined by setting
	\begin{equation*}
		d_{\mathcal M_0}(m_1\mu_1,m_2\mu_2):=|m_1-m_2|+\min\{f(m_1),f(m_2)\}d_\mathcal P(\mu_1,\mu_2),
	\end{equation*}
	for every $m_1,m_2\in \R_{>0}$ and $\mu_1,\mu_2\in\mathcal P_c(X)$,
	and
	\[
	d_{\mathcal M_0}(0,m\mu)=d_{\mathcal M_0}(m\mu,0):=m,\qquad d_{\mathcal M_0}(0,0):=0,
	\]
	is a distance function on $\mathcal M_c(X)$.
\end{proposition}

Clearly, the distance function $d_{\mathcal M_0}$ defined in Proposition \ref{general_Teo} is such that
\begin{equation*}
	d_{\mathcal M_0}(m\mu_1,m\mu_2)=f(m)d_\mathcal P(\mu_1,\mu_2),
\end{equation*}
for every $m\in \R_{>0}$ and every $\mu_1,\mu_2\in \mathcal P_c(X)$.
Moreover,
\[
\sup_{\mu\in \mathcal P_c(X)}d_{\mathcal M_0}(m_1\mu,m_2\mu)=|m_1-m_2|,
\]
for every $m_1,m_2\in[0,+\infty)$, which is a stronger property than assumption (ii) of Theorem \ref{Teo_pWas}, and also implies \eqref{destr_mass_eq} for every $\mu\in\mathcal P_c(\R^n)$.

The rest of the paper is organized as follows.
In Section \ref{Gen_Sec}, we first prove some general results about distance functions defined on the direct product of two metric spaces. Then, in Section \ref{Proof_Sec}, we exploit these results to prove our main theorems, stated here above in Section \ref{Main_res_Sect}. In Section \ref{last_Sect} we prove some further results, which are natural consequences of Theorem \ref{Teo_pWas}, and we provide some examples of meaningful distance functions defined on $\mathcal M_c^*(\R^n)$.

\section{Metrics on product spaces}\label{Gen_Sec}

We begin with the following result, which is a generalization of Theorem \ref{Teo_pWas} to product spaces $M=X\times Y$.

\begin{theorem}\label{Teo1}
	Let $(X,d_X)$ and $(Y,d_Y)$ be two metric spaces, and let $M:=X\times Y$ be equipped with a distance function $d_M:M\times M\to[0,+\infty)$ satisfying the following two properties:
	\begin{enumerate}[(i)]
		\item there exists a function $f:X\to(0,+\infty)$ such that
		\begin{equation}\label{fiber_cond_eq}
			d_M((x,y_1),(x,y_2))=f(x)d_Y(y_1,y_2),
		\end{equation}
		for every $x\in X$ and every $y_1,y_2\in Y$.
		\item For every $x_0\in X$ there exist a radius $r>0$ and a positive constant $C>0$, both depending only on $x_0$, such that
		\begin{equation}\label{unif_mass_cost_eq}
			\sup_{y\in Y}d_M((x_0,y),(x,y))\leq C
		\end{equation}
		for every $x\in B_r^X(x_0)$.
	\end{enumerate}
	If $X$ is connected and $\diam Y=+\infty$, then $f\equiv\lambda$, for some positive constant $\lambda>0$.
\end{theorem}

\begin{proof}
	Fix $x_0\in X$ and let $r$ and $C$ be as in point (ii). Notice that by \eqref{unif_mass_cost_eq} and the triangle inequality of $d_M$ we have
	\[
	\sup_{y\in Y}d_M((x_1,y),(x_2,y))\leq 2C,
	\]
	for every $x_1,x_2\in B_r^X(x_0)$.
	Thus, exploiting again the triangle inequality of $d_M$ and also \eqref{fiber_cond_eq}, we obtain
	\bgs{
		f(x_1)d_Y(y_1,y_2)&=d_M((x_1,y_1),(x_1,y_2))\\
		&\leq
		d_M((x_1,y_1),(x_2,y_1))+d_M((x_2,y_1),(x_2,y_2))+d_M((x_2,y_2),(x_1,y_2))\\
		&\leq
		4C+f(x_2)d_Y(y_1,y_2),
	}
	for every $x_1,x_2\in B_r^X(x_0)$ and every $y_1,y_2\in Y$. This implies that
	\[
	|f(x_1)-f(x_2)|d_Y(y_1,y_2)\leq 4C,
	\]
	and hence
	\[
	\sup_{x_1,x_2\in B_r^X(x_0)}|f(x_1)-f(x_2)|\leq \frac{4C}{\diam Y}.
	\]
	In particular, if $\diam Y=+\infty$ then $f$ must be constant in $B_r^X(x_0)$, hence it is locally constant in $X$. Therefore, if we also assume $X$ to be connected we conclude that $f$ is globally constant, as claimed.
\end{proof}

An example of such a distance function is given by
\[
d_M((x_1,y_1),(x_2,y_2)):= d_X(x_1,x_2)+ \lambda d_Y(y_1,y_2),
\]
for any fixed $\lambda>0$.

\begin{remark}
On the other hand, if $\diam Y<+\infty$, then we can easily find distances on $M$ satisfying points (i) and (ii) of Theorem \ref{Teo1}, with $f$ not constant. For example, let $f:X\to\R_{>0}$ be such that
\[
M:=\sup_{x\in X}f(x)<+\infty.
\]
Then, the function
\bgs{
	d_M((x_1,y_1),(x_2,y_2)):=\left\{\begin{array}{cc}f(x_1)d_Y(y_1,y_2) & \mbox{if }x_1=x_2,\\
		M\,\diam Y & \mbox{if }x_1\not=x_2,\end{array}\right.
}
is such a distance function.
\end{remark}

We now restrict our attention to the case in which $(X,d_X)=\R_{>0}$.

\begin{proposition}\label{bded_prop}
	Let $(Y,d_Y)$ be a metric space with $\diam Y<+\infty$ and let $f:\R_{>0}\to\R_{>0}$ be a Lipschitz, increasing, function such that
	\begin{equation}\label{global_Lip_hp}
		\sup_{\substack{x_1,x_2\in\R_{>0}\\ x_1\not=x_2}}\frac{|f(x_1)-f(x_2)|}{|x_1-x_2|}\leq\frac{1}{\diam Y}.
	\end{equation}
	Then, the function
	\bgs{
		d_M((x_1,y_1),(x_2,y_2)):=|x_1-x_2|+\min\{f(x_1),f(x_2)\}d_Y(y_1,y_2),
	}
	is a distance function on $M:=\R_{>0}\times Y$ satisfying point (i) of Theorem \ref{Teo1} and
	\begin{enumerate}
		\item[(ii')] for every $x_0\in \R_{>0}$ there exists a radius $r>0$ and a positive constant $C>0$, both depending only on $x_0$, such that
		\begin{equation*}
			\sup_{y\in Y}d_M((x_1,y),(x_2,y))\leq C |x_1-x_2|,
		\end{equation*}
		for every $x_1,x_2\in (x_0-r,x_0+r)$---which is more restrictive than point (ii) of Theorem \ref{Teo1}.
	\end{enumerate}
\end{proposition}

\begin{proof}
	We only need to verify that $d_M$ satisfies the triangle inequality. By symmetry we can assume without loss of generality that $x_1\leq x_2$, so that
	\[
	d_M((x_1,y_1),(x_2,y_2))=x_2-x_1+f(x_1)d_Y(y_1,y_2).
	\]
	Let us now consider a third point $(x_3,y_3)\in M$. If $x_3\geq x_1$, then
	\[
	f(x_1)\leq \min\{f(x_1),f(x_3)\}\quad\mbox{and}\quad f(x_1)\leq \min\{f(x_2),f(x_3)\},
	\]
	hence, exploiting also the triangle inequality of $|\,\cdot\,|$ and $d_Y$, we have
	\bgs{
		d_M((x_1,y_1)&,(x_2,y_2))=x_2-x_1+f(x_1)d_Y(y_1,y_2)\\
		&
		\leq
		x_3-x_1+|x_2-x_3|+f(x_1)d_Y(y_1,y_3)+f(x_1)d_Y(y_2,y_3)\\
		&
		\leq
		x_3-x_1+\min\{f(x_1),f(x_3)\}d_Y(y_1,y_3)+|x_2-x_3|+\min\{f(x_2),f(x_3)\}d_Y(y_2,y_3)\\
		&
		=d_M((x_1,y_1),(x_3,y_3))+d_M((x_2,y_2),(x_3,y_3)).
	}
	Suppose now that $x_3<x_1$, so that
	\[
	d_M((x_1,y_1),(x_3,y_3))=x_1-x_3+f(x_3)d_Y(y_1,y_3),
	\]
	and
	\[
	d_M((x_2,y_2),(x_3,y_3))=x_2-x_3+f(x_3)d_Y(y_2,y_3).
	\]
	Therefore the triangle inequality is equivalent to
	\begin{equation}\label{claim_triang}
		f(x_1)d_Y(y_1,y_2)-f(x_3)\big(d_Y(y_1,y_3)+d_Y(y_2,y_3)\big)\leq 2(x_1-x_3).
	\end{equation}
	Notice that by the triangle inequality of $d_Y$ we have
	\[
	-f(x_3)\big(d_Y(y_1,y_3)+d_Y(y_2,y_3)\big)\leq -f(x_3)d_Y(y_1,y_2).
	\]
	Thus, by exploiting also assumption \eqref{global_Lip_hp}, we obtain
	\bgs{
		f(x_1)d_Y(y_1,y_2)-f(x_3)\big(d_Y(y_1,y_3)+d_Y(y_2,y_3)\big)
		&
		\leq
		\big(f(x_1)-f(x_3)\big)d_Y(y_1,y_2)\\
		&
		\leq \frac{x_1-x_3}{\diam Y}\,\diam Y,
	}
	establishing \eqref{claim_triang} and concluding the proof of the Proposition.
\end{proof}

\begin{corollary}\label{cor_to0}
	Let $(Y,d_Y)$ be a metric space, $M:=\R_{>0}\times Y$, and let $M_0:=\{0\}\cup M$ be equipped with a distance function $d_{M_0}:M_0\times M_0\to[0,+\infty)$ such that the restriction of $d_{M_0}$ to $M$ satisfies points (i) and (ii) of Theorem \ref{Teo1}. If there exist at least two distinct points $p_1\not=p_2\in Y$ such that
		\begin{equation}\label{conv_mass0}
			\lim_{x\to0}d_{M_0}(0,(x,p_i))=0\qquad\mbox{for }i=1,2,
		\end{equation}
	then $\diam Y<+\infty$.
\end{corollary}

\begin{proof}
	We argue by contradiction and we assume that $\diam Y=+\infty$. Then, by Theorem \ref{Teo1} we know that
	\[
	d_{M_0}((x,y_1),(x,y_2))=\lambda d_Y(y_1,y_2),
	\]
	for every $x\in\R_{>0}$ and every $y_1,y_2\in Y$, for some fixed constant $\lambda>0$.
	Let $x_k\searrow0$. Then, by \eqref{conv_mass0} we have
	\bgs{
		0< d_Y(p_1,p_2)&=\frac{1}{\lambda}d_{M_0}((x_k,p_1),(x_k,p_2))\\
		&
		\leq\frac{1}{\lambda}\big(d_{M_0}((x_k,p_1),0)+d_{M_0}(0,(x_k,p_2))\big)\xrightarrow{k\to\infty}0,
	}
	giving a contradiction.
\end{proof}	

\begin{lemma}\label{bded_lem}
	Let $(Y,d_Y)$ be a metric space, $M:=\R_{>0}\times Y$, and let $M_0:=\{0\}\cup M$ be equipped with a distance function $d_{M_0}:M_0\times M_0\to[0,+\infty)$. Suppose that there exist $x_0\in\R_{>0}$ and $\lambda_0>0$ such that
	\[
	d_{M_0}((x_0,y_1),(x_0,y_2))=\lambda_0 d_Y(y_1,y_2),
	\]
	for every $y_1,y_2\in Y$, and
	\[
	\Lambda:=\sup_{y\in Y}d_{M_0}((x_0,y),0)<+\infty.
	\]
	Then $\diam Y<+\infty$.
\end{lemma}

\begin{proof}
	It is enough to notice that by the triangle inequality we have
	\bgs{
	d_Y(y_1,y_2)=\frac{1}{\lambda_0}d_{M_0}((x_0,y_1),(x_0,y_2))\leq
	\frac{1}{\lambda_0}\big(d_{M_0}((x_0,y_1),0)+d_{M_0}((x_0,y_2),0)\big),
}
	for every $y_1,y_2\in Y$.
	This implies that
	\[
	\diam Y\leq \frac{2\Lambda}{\lambda_0},
	\]
	concluding the proof.
\end{proof}

\begin{remark}\label{bded_rmk}
	If $\diam Y<+\infty$, then we can easily find many such distances $d_{M_0}$ on $M_0=\{0\}\cup\R_{>0}\times Y$.
	Indeed, let $f:\R_{>0}\to\R_{>0}$ be a Lipschitz, increasing, function such that
	\begin{equation*}
		\sup_{\substack{x_1,x_2\in\R_{>0}\\ x_1\not=x_2}}\frac{|f(x_1)-f(x_2)|}{|x_1-x_2|}\leq\frac{1}{\diam Y}\quad\mbox{and}\quad \lim_{x\searrow0}f(x)=0.
	\end{equation*}
	Define the function $d_{M_0}:M_0\times M_0\to[0,+\infty)$ by setting
	\bgs{
		d_{M_0}((x_1,y_1),(x_2,y_2)):=|x_1-x_2|+\min\{f(x_1),f(x_2)\}d_Y(y_1,y_2),
	}
	for every $(x_1,y_1),(x_2,y_2)\in \R_{>0}\times Y$
	and
	\[
	d_{M_0}(0,(x,y))=d_{M_0}((x,y),0):=x,\qquad d_{M_0}(0,0):=0.
	\]
	Then, by arguing as in Proposition \ref{bded_prop}, we can easily verify that $d_{M_0}$ is a distance function on $M_0$.
\end{remark}

\subsection{Proofs of the main results}\label{Proof_Sec}

We apply the results that we obtained in the previous Section to prove Theorem \ref{Teo_pWas}, Corollary \ref{Cor_pWas} and Proposition \ref{general_Teo}.

\begin{proof}[Proof of Theorem \ref{Teo_pWas}]
	We consider the metric space $(Y,d_Y):=(S,W_p)$, where $S\subseteq\mathcal P_c(\R^n)$ is the subspace given by
	\[
	S:=\{\delta_x\,:\,x\in\Sigma\},
	\]
	and the subspace
	\[
	M:=\{m\delta_x\,:\,x\in\Sigma,m>0\}\simeq\R_{>0}\times S\subseteq\mathcal M^*_c(\R^n).
	\]
	Then, $(M,d_\mathcal M)$ satisfies the hypothesis of Theorem \ref{Teo1}. Since $W_p(\delta_x,\delta_y)=|x-y|$ and $\Sigma$ is unbounded, also the space $(Y,d_Y)$ is not bounded. Thus, by Theorem \ref{Teo1} we conclude that $f\equiv\lambda$, as claimed.
\end{proof}

We observe that Corollary \ref{Cor_pWas}, with hypothesis \eqref{bded_0_eq}, means that we cannot extend a distance function $d_\mathcal M$ as in Theorem \ref{Teo_pWas} to a distance on the whole space $\mathcal M_c(\R^n)$, while preserving the boundedness assumption of point (ii) also around the zero measure.
Similarly, such an extension is not possible under the very weak continuity hypothesis \eqref{destr_mass_eq}---which, roughly speaking, translates the fact that when we fix any probability measure $\mu$ and we decrease its mass, considering $m\mu$, with $m\searrow0$, we expect to be ``approaching'' the zero measure. This can be seen, e.g., as a compatibility of $d_{\mathcal M_0}$ with the weak convergence $m\mu\rightharpoonup0$ as $m\searrow0$. A related result is given in Proposition \ref{isom_inv_dist_prop}, assuming invariance of the distance with respect to isometries.

\begin{proof}[Proof of Corollary \ref{Cor_pWas}]
	The claims of the Corollary follow by Lemma \ref{bded_lem} and by arguing as in the proof of Corollary \ref{cor_to0}.
\end{proof}

\begin{proof}[Proof of Proposition \ref{general_Teo}]
	The proof follows from Proposition \ref{bded_prop} and Remark \ref{bded_rmk}.	
\end{proof}

\section{Further results and examples}\label{last_Sect}

In this Section, we discuss other obstructions to the extension of the Wasserstein distances to general measures, based on invariance with respect to isometries.

We prove that there exists no distance function on $\mathcal M_c(\R^n)$ which is invariant with respect to isometries and which satisfies \eqref{fiber_condition_Wp} even on a single fiber $\{m_0\}\times\mathcal P_c(\R^n)$. More precisely:

\begin{proposition}[Unbounded space with zero, isometries invariance]\label{isom_inv_dist_prop}
	Let $p\in[1,+\infty)$. There exists no distance function $d$ on $\mathcal M_c(\R^n)$ such that
	\[
	d(\mu_1,\mu_2)=\lambda_0 W_p\Big(\frac{\mu_1}{m_0},\frac{\mu_2}{m_0}\Big),
	\]
	whenever $|\mu_1|=|\mu_2|=m_0$, for some $m_0\in\R_{>0}$ and $\lambda_0>0$, and which is invariant with respect to the isometries of $\R^n$, i.e.
	\begin{equation}\label{inv_eq}
		d(T\#\mu_1,T\#\mu_2)=d(\mu_1,\mu_2)
	\end{equation}
	for every isometry $T:\R^n\to\R^n$ and for every $\mu_1,\mu_2\in\mathcal M_c(\R^n)$.
\end{proposition}

\begin{proof}
	Notice that for every $x\in\R^n$ the function $\tau_x:\R^n\to\R^n$ defined by $\tau_x(y):=y+x$ is an isometry, and $\delta_x=\tau_x\#\delta_0$.
	By \eqref{inv_eq} we thus obtain
	\bgs{
		d(m_0\delta_x,0)=d(m_0\tau_x\#\delta_0,\tau_x\#0)=d(\tau_x\#(m_0\delta_0),\tau_x\#0)=d(m_0\delta_0,0),
	}
	for every $x\in \R^n$. The conclusion then follows from Lemma \ref{bded_lem}, by considering $(Y,d_Y)=(S,W_p)$, with $S:=\{\delta_x\,:\,x\in\R^n\}$, which is not bounded, and $d_{M_0}=d$ on $M_0:=\{0\}\cup\R_{>0}\times S\subseteq\mathcal M_c(\R^n)$.
\end{proof}

Once again, in Proposition \ref{isom_inv_dist_prop} the obstruction to the existence of such a distance function $d$ is given by the combination of the presence of the zero measure with the unboundedness of $W_p$.
On a related note, we also observe that hypothesis (ii) in Theorem \ref{Teo_pWas} is weaker than the combination of other very natural assumptions, like the invariance with respect to isometries and the compatibility with weak convergence when varying the mass. Thus, we obtain the following:

\begin{proposition}[Unbounded space, isometries invariance]\label{final_prop}
	Let $p\in[1,+\infty)$ and let $d_\mathcal M:\mathcal M^*_c(\R^n)\times \mathcal M^*_c(\R^n)\to[0,+\infty)$ be a distance function such that
	\begin{equation*}
		d_\mathcal M(m\mu_1,m\mu_2)=f(m)W_p(\mu_1,\mu_2),
	\end{equation*}
	for every $m\in \R_{>0}$ and every $\mu_1,\mu_2\in \mathcal P_c(\R^n)$, for some function $f:\R_{>0}\to \R_{>0}$. Suppose that $d_\mathcal M$ is invariant with respect to the isometries of $\R^n$, i.e.
	\begin{equation}\label{hp1_prp}
		d_\mathcal M(T\#\mu_1,T\#\mu_2)=d_\mathcal M(\mu_1,\mu_2)
	\end{equation}
	for every isometry $T:\R^n\to\R^n$ and for every $\mu_1,\mu_2\in\mathcal M_c^*(\R^n)$. Assume moreover that
	\begin{equation}\label{hp2_prp}
		\lim_{|m-m_0|\to0}d_\mathcal M(m\mu,m_0\mu)=0,
	\end{equation}
	for every $m_0\in\R_{>0}$ and every $\mu\in\mathcal P_c(\R^n)$. Then, $f\equiv\lambda$, for some $\lambda>0$.	
\end{proposition}

\begin{proof}
	Fix any $m_0\in\R_{>0}$. By the continuity hypothesis in \eqref{hp2_prp}, we know in particular that there exists $r>0$, depending only on $m_0$, such that
	\[
	d_\mathcal M(m\delta_0,m_0\delta_0)\leq 1,
	\]
	for every $m\in(m_0-r,m_0+r)$. Moreover, exploiting the isometric invariance \eqref{hp2_prp} and arguing as in the proof of Proposition \ref{isom_inv_dist_prop}, we have
	\bgs{
		d_\mathcal M(m\delta_x,m_0\delta_x)&=d_\mathcal M(m\tau_x\#\delta_0,m_0\tau_x\#\delta_0)=d_\mathcal M(\tau_x\#(m\delta_0),\tau_x\#(m_0\delta_0))\\
		&
		=d_\mathcal M(m\delta_0,m_0\delta_0),
	}
	for every $x\in\R^n$. Therefore, $d_\mathcal M$ satisfies hypothesis (ii) in Theorem \ref{Teo_pWas}, with $\Sigma=\R^n$, and the conclusion follows.
\end{proof}

\subsection{Examples}
The simplest examples of distance functions defined on $\mathcal M_c^*(\R^n)$ which satisfy \eqref{fiber_condition_Wp} (with $f$ constant) are given by product metrics. Indeed, if $d$ is any distance function on $(0,+\infty)$, then
\[
d_\mathcal M(\mu_1,\mu_2):=\sqrt{d^2(|\mu_1|,|\mu_2|)+\lambda^2 W_p^2\Big(\frac{\mu_1}{|\mu_1|},\frac{\mu_2}{|\mu_2|}\Big)},
\]
is such a distance function.
Of particular interest are the cases given in Example \ref{exe_prod} and Example \ref{discont_exe} here below, which also satisfy the boundedness condition \eqref{in_unif_mass_cost_eq}.

\begin{example}\label{exe_prod}
	Let $q\in[1,+\infty)$ and define
	\[
	d_{\mathcal M,q}(\mu_1,\mu_2):=\left(\big||\mu_1|-|\mu_2|\big|^q+\lambda^q W_p^q\Big(\frac{\mu_1}{|\mu_1|},\frac{\mu_2}{|\mu_2|}\Big)\right)^\frac{1}{q},
	\]
	and also
	\[
	d_{\mathcal M,\infty}(\mu_1,\mu_2):=\max\left\{\big||\mu_1|-|\mu_2|\big|,\lambda W_p\Big(\frac{\mu_1}{|\mu_1|},\frac{\mu_2}{|\mu_2|}\Big)\right\},
	\]
	for every $\mu_1,\mu_2\in\mathcal M_c^*(\R^n)$. Then, $d_{\mathcal M,q}$ is a distance function on $\mathcal M_c^*(\R^n)$ satisfying points (i) and (ii) of Theorem \ref{Teo_pWas}, for every $q\in[1,+\infty]$. Actually, each $d_{\mathcal M,q}$ is such that
	\[
	\sup_{\mu\in\mathcal P_c(\R^n)}d_{\mathcal M,q}(m_0\mu,m\mu)= |m_0-m|,
	\]
	for every $m,m_0\in\R_{>0}$, which is more restrictive than point (ii).
\end{example}

\begin{example}\label{discont_exe}
	Let $d$ be a bounded distance function on $(0,+\infty)$, with $d(m_1,m_2)\leq C_d$ for every $m_1,m_2\in(0,+\infty)$. Then,
	\[
	d_\mathcal M(\mu_1,\mu_2):=d(|\mu_1|,|\mu_2|)+\lambda W_p\Big(\frac{\mu_1}{|\mu_1|},\frac{\mu_2}{|\mu_2|}\Big),
	\]
	is a distance function on $\mathcal M^*_c(\R^n)$, satisfying point (i) of Theorem \ref{Teo_pWas} and the global boundedness property
	\[
	\sup_{\mu\in\mathcal P_c(\R^n)}d_\mathcal M(m_0\mu,m\mu)\leq C_d,
	\]
	for every $m_0,m\in(0,+\infty)$, which is clearly stronger than point (ii). In particular, by considering as $d$ the discrete metric on $(0,+\infty)$, we obtain an example of distance function $d_\mathcal M$ for which we do not have continuity in the mass, in the sense that
	\begin{equation*}
		\lim_{|m- m_0|\to0}d_\mathcal M(m_0\mu,m\mu)=0,
	\end{equation*}
 does not hold true for any $\mu\in\mathcal P_c(\R^n)$.
\end{example}

Other less trivial examples can be obtained as the so-called warped product metrics, by exploiting the fact that $(\mathcal P_c(\R^n),W_p)$ is a geodesic space, see \cite{villani2}.

\begin{example}[Warped products]\label{warped_exe}
	Let $d_{\mathcal M,2}$ be the product metric on $\mathcal M_c^*(\R^n)$ defined in Example \ref{exe_prod}. Given a Lipschitz continuous curve $\gamma:I\to(\mathcal M_c^*(\R^n),d_{\mathcal M,2})$, where $I:=[0,1]\subseteq\R$, we denote by $m_\gamma:I\to\R_{>0}$ and $\rho_\gamma:I\to(\mathcal P_c(\R^n),W_p)$ its projections, defined respectively as
	\[
	m_\gamma(t):=|\gamma(t)|\quad\mbox{and}\quad\rho_\gamma(t):=\frac{\gamma(t)}{|\gamma(t)|},
	\]
	which are Lipschitz continuous curves. Let $\Gamma$ denote the set of all such curves.
	Given a continuous function $g:(\mathcal P_c(\R^n),W_p)\to\R_{>0}$, we define the warped product metric, with warping function $g$, as
	\[
	d_{\mathcal M,g}(\mu_1,\mu_2):=\inf_{\substack{\gamma\in\Gamma\\ \gamma(0)=\mu_1,\gamma(1)=\mu_2}}\int_I\sqrt{g^2(\rho_\gamma(t))|m_\gamma'(t)|^2+|\rho_\gamma'(t)|^2}\,dt.
	\]
	Here above we have denoted, with a slight abuse of notation,
	\[
	|\rho_\gamma'(t)|:=\lim_{s\to0}\frac{W_p(\rho_\gamma(t),\rho_\gamma(t+s))}{|s|},
	\]
	which exists for almost every $t\in I$, since $\rho_\gamma$ is Lipschitz, see, e.g., \cite[Theorem 2.7.6]{BBI}. This definition can be found, e.g., in \cite[Section 3.6.4]{BBI}, and it is easy to check that it coincides with the one given in \cite[Section 3.1]{Warped}.
	By \cite[Proposition 3.1]{Warped}, $d_{\mathcal M,g}$ is a distance function on $\mathcal M_c^*(\R^n)$, and, by \cite[Lemma 3.2]{Warped}, it satisfies \eqref{fiber_condition_Wp}, as indeed
	\begin{equation}\label{eqn_wrp}
	d_{\mathcal M,g}(\mu_1,\mu_2)=W_p\Big(\frac{\mu_1}{|\mu_1|},\frac{\mu_2}{|\mu_2|}\Big),
	\end{equation}
	for every $\mu_1,\mu_2\in\mathcal M_c^*(\R^n)$ such that $|\mu_1|=|\mu_2|$. It is easy to find examples of warping functions $g$ for which the corresponding distance $d_{\mathcal M,g}$ does not satisfy point (ii) of Theorem \ref{Teo_pWas}, such as $g(\mu):=1+W_p(\mu,\delta_0)$. Similarly, if we consider the warping function
	\begin{equation*}
		g(\mu):=1+\inf_{x\in\R^n}W_p(\mu,\delta_x),
	\end{equation*}
	then, $d_{\mathcal M,g}$ satisfies point (ii) of Theorem \ref{Teo_pWas}, with $\Sigma=\R^n$. However, since we can find a sequence $\{\mu_k\}_{k\in\N}\subseteq\mathcal P_c(\R^n)$ such that $g(\mu_k)\geq k$ for every $k$, it is easy to verify that \eqref{in_unif_mass_cost_eq} does not hold on the whole of $\mathcal P_c(\R^n)$. That is, there exist no $m_0,r,C>0$ for which
	\[
	\sup_{\mu\in\mathcal P_c(\R^n)}d_{\mathcal M,g}(m_0\mu,m\mu)\leq C,
	\]
	for every $m\in(m_0-r,m_0+r)$.
	
	Furthermore, we observe that for every choice of warping function $g$, the distance $d_{\mathcal M,g}$ is compatible with weak convergence when varying the mass, as indeed
	\[
	d_{\mathcal M,g}(m_0\mu,m\mu)\leq g(\mu)|m-m_0|,
	\]
	for every $\mu\in\mathcal P_c(\R^n)$. Nevertheless, since \eqref{eqn_wrp} holds true, by arguing as in the proof of Corollary \ref{cor_to0} we see that no such distance can be extended to a distance $d$ defined on the whole of $\mathcal M_c(\R^n)$ and with the property that
	\[
	\lim_{m\searrow0}d(0,m\mu_i)=0,
	\]
	for at least two different $\mu_1\not=\mu_2\in\mathcal P_c(\R^n)$.
\end{example}

\bibliographystyle{amsplain}
\bibliography{biblio}

\end{document}